\documentclass[12pt]{article}
\usepackage[utf8]{inputenc}
\usepackage{amsmath}
\usepackage{amssymb}
\usepackage{mathrsfs}
\usepackage{amsthm}
\usepackage{cite}
\usepackage{hyperref}
\usepackage{enumerate}
\usepackage{combelow}
\usepackage[left=0.7in,right=0.7in,top=0.7in,bottom=0.7in]{geometry}
\usepackage{titlesec}
\titleformat*{\section}{\large\bfseries}
\newtheorem{theorem}{Theorem}[section]
\newtheorem{lemma}[theorem]{Lemma}
\newtheorem{corollary}[theorem]{Corollary}
\newtheorem{definition}[theorem]{Definition}
\newtheorem{example}[theorem]{Example}

\newtheorem{remark}[theorem]{Remark}
\numberwithin{equation}{section}
\DeclareMathOperator{\taninv}{tan^{-1}}
\title{Common Fixed Point results in Complex valued metric spaces via simulation functions}
\author{\large Anuradha Gupta$^1$ and Manu Rohilla$^2$\\ $^1$ {\small Department of Mathematics, Delhi College of Arts and Commerce,}\\ {\small University of Delhi, Netaji Nagar, New Delhi-110023, India.}\\{\small E-mail:dishna2@yahoo.in}\\{\small $^2$Department of Mathematics, University of Delhi, New Delhi-110007, India.}\\{\small E-mail:manurohilla25994@gmail.com}}
\date{}
\begin{document}
\maketitle

\begin{abstract}
In this paper, the notion of $\mathbb{C}$-simulation function is introduced and the existence and uniqueness of common fixed points of two self-mappings satisfying contractive conditions in the setting of complex valued metric spaces  via $\mathbb{C}$-simulation functions  are studied. Examples are also provided to demonstrate the results. The existence and uniqueness of a first-order periodic differential equation is also obtained as an application of the result.
\end{abstract}

\section{Introduction and Preliminaries}
In 2011, Azam et.al \cite{2} introduced the concept of complex valued metric spaces which is a generalization of metric spaces and established common fixed point of self-mappings satisfying generalized contractive condition. Many authors (see \cite{2,7,9,10,11}) have obtained fixed point results in complex valued metric spaces and proved the existence and uniqueness of solution of nonlinear integral equations.

Let $\mathbb{N}$, $\mathbb{R}$ and $\mathbb{C}$ be the set of natural numbers, real numbers and complex numbers, respectively. For $z_1,z_2 \in \mathbb{C}$, define a partial order $\precsim$ on $\mathbb{C}$ as follows:
$$z_1 \precsim z_2 \quad \mbox{if and only if} \quad \mbox{Re}z_1 \leq \mbox{Re}z_2 \thinspace \thinspace \mbox{and} \thinspace \thinspace \mbox{Im}z_1 \leq \mbox{Im}z_2.$$
In other words, $z_1 \precsim z_2$ if one of the following conditions holds:

(i) Re($z_1)=$ Re($z_2$) and Im($z_1)=$ Im($z_2$),

(ii) Re($z_1)=$ Re($z_2$) and Im($z_1)<$ Im($z_2$),

(iii) Re($z_1)<$ Re($z_2$) and Im($z_1)=$ Im($z_2$),

(iv) Re($z_1)<$ Re($z_2$) and Im($z_1)<$ Im($z_2$).\\
We will write $z_1 \precnsim z_2$ if $z_1 \neq z_2$ and any of (ii), (iii) or (iv) is satisfied. If only (iv) is satisfied, then we write $z_1 \prec z_2$. If $a,b \in \mathbb{R}$ and $a <b$ then $az\precnsim bz$ for all $z \in \mathbb{C}$. It is easily observed that if $0\precsim z_1 \precnsim z_2$ then $\vert z_1 \vert < \vert z_2 \vert$. Also, if $z_1 \precsim z_2$ and $z_2 \prec z_3$ then $z_1 \prec z_3$.

Azam et.al \cite{2} defined complex valued metric as follows:
\begin{definition}
\emph{A complex valued metric on a nonempty set $X$ is a mapping $d:X \times X \rightarrow \mathbb{C}$ such that for all $x,y,z \in X$, the following conditions are satisfied:
\begin{enumerate}[(i)]
\item $0 \precsim d(x,y)$ and $d(x,y)=0$ if and only if $x=y$,
\item $d(x,y)=d(y,x)$
\item $d(x,y) \precsim d(x,z)+d(z,y)$.
\end{enumerate}
The pair $(X,d)$ is called a complex valued metric space.}
\end{definition}
\begin{example}
\emph{(\cite{9}, \cite{10})Let $X= \mathbb{C}$ and $d_j:X \times X \rightarrow \mathbb{C}$, $j=1,2,3$ be defined as
\begin{enumerate}[(i)]
\item $d_1(z_1,z_2)= \vert z_1-z_2\vert$ for all $z_1,z_2 \in \mathbb{X}$.
\item $d_2(z_1,z_2)= e^{ik} \vert z_1-z_2\vert$ for all $z_1,z_2 \in \mathbb{X}$ and $k \in \mathbb{R}$.
\item $d_3(z_1,z_2)= \vert x_1-x_2\vert +i\vert y_1-y_2\vert $ for all $z_1=x_1+iy_1,$ $z_2=x_2+iy_2 \in \mathbb{X}$.
\end{enumerate}
Then $(X,d_j)$ ($j=1,2,3$) is a complex valued metric space.  }
\end{example}
\begin{definition}
\emph{(\cite{2}) Let $(X,d)$ be a complex valued metric space and $A \subset X$. Then\\
(i) A point $a \in A$ is called an interior point of $A$ if there exists $0 \prec r \in \mathbb{C}$ such that $$B(a,r)=\{y \in X : d(a,y) \prec r \} \subset A.$$
(ii) A point $a \in A$ is called a limit point of $A$ if for every $0 \prec r \in \mathbb{C}$ we have $$ B(a,r) \cap (A \setminus \{a\}) \neq \emptyset . $$
(iii) $A$ is called open if each point of $A$ is a limit point of $A$. $A$ is called closed if each limit point of $A$ belongs to $A$.\\
(iv) The family $\mathfrak{B}=\{B(x,r):x \in X, 0\prec r \}$ is a sub-basis for a Hausdorff topology $\tau$ on $X$.}
\end{definition}
\begin{definition}
\emph{(\cite{2}) Let $(X,d)$ be a complex valued metric space. Let $\{x_n\}$ be a sequence in $X$ and $x \in X$. Then}\\
\emph{(i) the sequence $\{x_n\}$} converges \emph{$x$ if for every $0 \prec c \in \mathbb{C}$ there exists $n_0 \in \mathbb{N}$ such that $d(x_n,x) \prec c$ for all $n >n_0$. We denote this by $\lim\limits_{n \rightarrow \infty} x_n =x$.}\\
\emph{(ii) the sequence $\{x_n\}$ is called} Cauchy \emph{in $(X,d)$ if for every $0 \prec c \in \mathbb{C}$ there exists $n_0 \in \mathbb{N}$ such that $d(x_n,x_m) \prec c$ for all $n,m >n_0$.}\\
\emph{(iii) the space $(X,d)$ is said to be} complete \emph{if every Cauchy sequence is convergent.}
\end{definition}
\begin{lemma} \emph{(See \cite{2})}
Let $(X,d)$ be a complex valued metric space and $\{x_n\}$ be a sequence in $X$. Then $\{x_n\}$ converges to $x$ if and only if $\vert d(x_n,x) \vert \rightarrow 0$ as $n \rightarrow \infty$.
\end{lemma}
\begin{lemma} \emph{(See \cite{2})}
Let $(X,d)$ be a complex valued metric space and $\{x_n\}$ be a sequence in $X$. Then $\{x_n\}$ is a Cauchy sequence  if and only if $\vert d(x_n,x_{n+m}) \vert \rightarrow 0$ as $n \rightarrow \infty$.
\end{lemma}
\begin{definition}
\emph{Let $S$ and $T$ be self-mappings of a nonempty set $X$. A point $x \in X$ is called common fixed point of $S$ and $T$ if $Sx=Tx$.}
\end{definition}
Recently, Khojasteh et.al \cite{5} introduced the notion of simulation functions. Using this concept several fixed point results of the literature have been generalized and unified. Later, this idea have been undertaken by several authors to prove fixed point results in the framework of metric spaces and its various generalizations. Khojasteh et al. defined simulation function as follows:
\begin{definition}
\emph{(\cite{5}) A simulation function is a mapping $\zeta: [0, \infty) \times [0, \infty) \rightarrow \mathbb{R}$ satisfying the following conditions:}

\emph{(i) $\zeta(0,0)=0$,}

\emph{(ii) $\zeta(t,s)<s-t$ for all $t,s>0$,}

\emph{(iii)  if $\{t_n\}$ and $\{s_n\}$ are sequences in $(0,\infty)$ such that $\lim\limits_{n \rightarrow \infty} t_n= \lim\limits_{n \rightarrow \infty} s_n>0$, then  $$\limsup\limits_{n \rightarrow \infty} \zeta(t_n,s_n)<0.$$}
\end{definition}
In this paper, we introduce the concept of $\mathbb{C}$-simulation function. Common fixed point results of self mappings in the context of complete complex valued metric spaces via simulation functions are obtained. Examples are also provided to illustrate the applicability of the results obtained. In the last section, a first-order periodic differential  equation is solved as an application of one of the results of the main section.
\section{Main Results}
In this section, we define $\mathbb{C}$-simulation function and give some examples. Also, we prove some common fixed point results. Let $T:X \rightarrow X$ be a map, denote \begin{align*}
\mbox{Fix}(T)&=\{x \in X:Tx=x\}\\
\mathbb{S}&=\{z \in \mathbb{C}: 0\precsim z\}.
\end{align*}
\begin{definition}\label{definition2.1}
\emph{A mapping $\xi : \mathbb{S} \times \mathbb{S} \rightarrow \mathbb{C}$ is called a $\mathbb{C}$\emph{-simulation function} if it satisfies the following conditions:
\begin{enumerate}[(i)]
\item $\xi (0,0)=0$,
\item $\xi (t,s) \precnsim s-t$ for all $0\precnsim t,s$,
\item  if $\{t_n \}$ and $\{s_n\}$ are sequences in $\mathbb{S} \setminus \{0\}$ such that $0\precnsim \lim\limits_{n \rightarrow \infty} \vert t_n \vert=\lim\limits_{n \rightarrow \infty} \vert s_n \vert$ then $$\limsup\limits_{n \rightarrow \infty} \xi(\vert t_n \vert, \vert s_n \vert )\precnsim 0.$$
\end{enumerate}}
\end{definition}
Throughout this paper, we denote by $\mathscr{Z}_{\mathbb{C}}$ the collection of all $\mathbb{C}$-simulation functions.
\begin{example}
\emph{Let $\xi_j:\mathbb{S} \times \mathbb{S} \rightarrow \mathbb{C}$, $j=1,2,3$ be defined as
\begin{enumerate}[(i)]
\item $\xi_{1}(t,s)= \lambda s-t$ for all $t,s \in \mathbb{S}$, where $0<\lambda<1$.
\item $\xi_{2}(t,s)= \psi(s)-\phi(t)$ for all $t,s \in S$, where $\psi,\phi:\mathbb{S} \rightarrow \mathbb{S}$ are continuous functions satisfying $\psi(t) \precnsim t \precsim \phi(t)$.
\item $\xi_{3}(t,s)=s-t-i \vert t \vert $ for all $t,s \in \mathbb{S}$.
\end{enumerate}
It is easy to see that each $\xi_j$ $(j=1,2,3)$ is a $\mathbb{C}$-simulation function.}
\end{example}
Samet et.al \cite{8} introduced the concept of $\alpha$-admissible maps and Abdeljawad \cite{1} suggested the notion of $\alpha$-admissiblity for a pair of mappings. Motivated by them we introduce the concept of $\alpha_{\mathbb{C}}$-admissible mappings.
\begin{definition}
\emph{Let $\alpha :X \times X \rightarrow \mathbb{S}$ and $S,T : X \rightarrow X$, then the mapping $T$ is called $\alpha_{\mathbb{C}}$-admissible if for all $x,y \in X$ we have $$1\precsim \alpha(x,y) \quad \mbox{implies} \quad 1\precsim \alpha(Tx,Ty)$$ and the pair $(S,T)$ is $\alpha_{\mathbb{C}}$-admissible if 
$$ 1\precsim \alpha(x,y) \quad \mbox{implies} \quad 1\precsim \alpha(Sx,Ty) \thinspace \thinspace \mbox{and} \thinspace \thinspace 1\precsim \alpha(Tx,Sy).$$}
\end{definition}
\begin{definition}
\emph{Let $\alpha :X \times X \rightarrow \mathbb{S}$ and $S,T : X \rightarrow X$, then the pair $(S,T)$ is $\alpha_{\mathbb{C}}$\emph{-orbital admissible} if
$$ 1\precsim \alpha(x,Sx) \quad \mbox{implies} \quad 1\precsim \alpha(Sx,TSx) \thinspace \thinspace \mbox{and} \thinspace \thinspace 1\precsim \alpha(Tx,S^2x)$$
and 
$$ 1\precsim \alpha(x,Tx) \quad \mbox{implies} \quad 1\precsim \alpha(Sx,T^2x) \thinspace \thinspace \mbox{and} \thinspace \thinspace 1\precsim \alpha(Tx,STx).$$ Furthermore, the pair $(S,T)$ is \emph{triangular $\alpha_{\mathbb{C}}$-orbital admissible} if
$$ 1\precsim \alpha(x,y) \thinspace \thinspace \mbox{and} \thinspace 1\precsim \alpha(y,Sy) \quad \mbox{implies} \quad 1\precsim \alpha(x,Sy)$$
and 
$$ 1\precsim \alpha(w,z) \thinspace \thinspace \mbox{and} \thinspace 1\precsim \alpha(z,Tz) \quad \mbox{implies} \quad 1\precsim \alpha(w,Tz).$$}
\end{definition}
\begin{definition}
\emph{Let $(X,d)$ be a complex valued metric space, then $X$ is said to be $\alpha_{\mathbb{C}}$\emph{-regular} if for every sequence $\{x_n\} \subset X$ such that $x_n \rightarrow x$ in $(X,d)$ there exists a subsequence $\{x_{n_k}\}$ of $\{x_n\}$ such that
$$ 1\precsim \alpha(x_{n_k},x) \thinspace \thinspace \mbox{and} \thinspace \thinspace 1\precsim \alpha(x,x_{n_k}).$$}
\end{definition}
\begin{definition}
\emph{Let $(X,d)$ be a complex valued metric space, $S,T: X \rightarrow X$ be self-mappings and $\xi \in \mathscr{Z}_{\mathbb{C}}$. Then the pair $(S,T)$ is called an $\alpha_{\mathbb{C}}$\emph{-admissible $\mathscr{Z}_{\mathbb{C}}$-contraction} with respect to $\xi$ if for all $x,y \in X$, the following conditions are satisfied:
\begin{enumerate}[(i)]
\item $0\precsim \alpha(x,y) d(Sx,Ty)$,
\item $0\precsim \xi(\alpha(x,y) d(Sx,Ty),d(x,y))$,
\item $0\precsim \xi(\vert \alpha(x,y) d(Sx,Ty) \vert,\vert d(x,y) \vert )$.
\end{enumerate}} 
\end{definition}
\begin{theorem}\label{theorem1}
Let $(X,d)$ be a complete complex valued metric space and $S,T:X \rightarrow X$ be self-mappings such that the pair $(S,T)$ is an $\alpha_{\mathbb{C}}$-admissible $\mathscr{Z}_{\mathbb{C}}$-contraction with respect to $\xi$. Suppose that 

(i) $(S,T)$ is triangular $\alpha_{\mathbb{C}}$-orbital admissible,

(ii) there exists $x_0 \in X$ such that $1 \precsim \alpha(x_0,Sx_0)$ and $1\precsim \alpha(Sx_0,x_0)$,

(iii) $X$ is $\alpha_{\mathbb{C}}$-regular.\\
Moreover, if for $x,y \in$ Fix$(T)$ $\cap$ Fix$(S)$ we have $1\precsim \alpha(x,y)$. Then $S$ and $T$ have a unique common fixed point in $X$.
\end{theorem}
\begin{proof}
Construct a sequence $\{x_n\}$ in $X$ such that
$$x_{2n+1}=Sx_{2n}, \thinspace \thinspace x_{2n+2}=Tx_{n+1}, \thinspace \thinspace \mbox{for all} \thinspace n=0,1,2,\ldots .$$
Since $(S,T)$ is $\alpha_{\mathbb{C}}$-orbital admissible, $1 \precsim \alpha(x_0,Sx_0)$ and $1\precsim \alpha(Sx_0,x_0)$ then $1 \precsim \alpha(x_n,x_{n+1})$ and $1\precsim \alpha(x_{n+1},x_n)$. For all $n \geq 0$ we have
\begin{align*}
0& \precsim \xi(\alpha(x_{2n},x_{2n+1}) d(Sx_{2n}, Tx_{2n+1}), d(x_{2n},x_{2n+1}))\\
& \precnsim d(x_{2n},x_{2n+1})-\alpha(x_{2n},x_{2n+1}) d(x_{2n+1}, x_{2n+2}).
\end{align*}
Therefore, $\vert d(x_{2n+1}, x_{2n+2}) \vert \leq \vert \alpha(x_{2n},x_{2n+1}) d(x_{2n+1}, x_{2n+2}) \vert < \vert d(x_{2n},x_{2n+1}) \vert $. Also,
\begin{align*}
0& \precsim \xi(\alpha(x_{2n+2},x_{2n+1})d(Sx_{2n+2},Tx_{2n+1}),d(x_{2n+2},x_{2n+1}))\\
& \precnsim d(x_{2n+2},x_{2n+1})-\alpha(x_{2n+2},x_{2n+1})d(x_{2n+3},x_{2n+2})
\end{align*}
which implies that $\vert d(x_{2n+2},x_{2n+3}) \vert < \vert d(x_{2n+1},x_{2n+2}) \vert $. Then $\{\vert d(x_n,x_{n+1}) \vert \}$ is a monotonically non-increasing sequence of non-negative real numbers therefore, it is convergent. Let $\lim\limits_{n \rightarrow \infty} \vert d(x_n,x_{n+1}) \vert =L \geq 0$. Therefore, $\lim\limits_{n \rightarrow \infty} \vert d(x_{2n+1},x_{2n+2})\vert$ $=L$. Suppose that $L>0$. Also, $$0\precsim \xi (\alpha(x_{2n},x_{2n+1}) d(x_{2n+1},x_{2n+2}),d(x_{2n},x_{2n+1}))$$ which gives $\vert d(x_{2n+1},x_{2n+2}) \vert \leq \vert \alpha(x_{2n},x_{2n+1})d(x_{2n+1},x_{2n+2}) \vert < \vert d(x_{2n},x_{2n+1}) \vert$. Therefore, $$\lim\limits_{n \rightarrow \infty} \vert \alpha(x_{2n},x_{2n+1})d(x_{2n+1}, x_{2n+2}) \vert =L.$$ Then using (iii) of Definition \ref{definition2.1} we get,  $0\precsim \xi (\vert \alpha(x_{2n},x_{2n+1}) d(x_{2n+1},x_{2n+2})\vert , \vert d(x_{2n},x_{2n+1})\vert) \precnsim 0$, a contradiction. Therefore, 
\begin{equation}\label{eq1}
\lim\limits_{n\rightarrow \infty}\vert d(x_n,x_{n+1}) \vert =0. 
\end{equation}
Now we prove that $\{x_n \}$ is a Cauchy sequence in $(X,d)$. It suffices to show that $\{x_{2n}\}$ is a Cauchy sequence in $(X,d)$. On the contrary, assume that $\{x_{2n}\}$ is not a Cauchy sequence. Then there exists $c \in \mathbb{C}$ with $0 \prec c$ for which we can find two subsequences $\{x_{2m_i}\}$ and $\{x_{2n_i}\}$ of $\{x_{2n}\}$ such that $n_i$ is the smallest positive integer with 
$$2n_i > 2m_i \geq i \quad \mbox{and} \quad c \precsim d(x_{2m_i},x_{2n_i}).$$
This means that  $ d(x_{2m_i},x_{2n_i-2})\prec c $. Consider
\begin{align*}
c &\precsim  d(x_{2m_i},x_{2n_i}) \\
& \precsim d(x_{2m_i},x_{2n_i-1})+d(x_{2n_i-1},x_{2n_i})\\
& \precsim  d(x_{2m_i},x_{2n_i-2}) +  d(x_{2n_i-2},x_{2n_i-1})+ d(x_{2n_i-1},x_{2n_i})
\end{align*} 
which implies that 
$$ \vert c \vert \leq \vert d(x_{2m_i},x_{2n_i}) \vert < \vert c\vert + \vert d(x_{2n_i-2},x_{2n_i-1})\vert+\vert d(x_{2n_i-1},x_{2n_i}) \vert.$$
Letting $i \rightarrow \infty$ and using (\ref{eq1}) we get, $\lim\limits_{i \rightarrow \infty} \vert d(x_{2m_i},x_{2n_i}) \vert = \vert c \vert$ and  $\lim\limits_{i \rightarrow \infty} \vert d(x_{2m_i},x_{2n_i-1}) \vert= \vert c \vert$. Also, 
\begin{align*}
c & \precsim d(x_{2m_i},x_{2n_i})\\
& \precsim d(x_{2m_i},x_{2m_i+1})+d(x_{2m_i+1},x_{2n_i})\\
 & \precsim  d(x_{2m_i},x_{2m_i+1})  + d(x_{2m_i+1},x_{2m_i})+  d(x_{2m_i},x_{2n_i}).
 \end{align*} This gives $\lim\limits_{i \rightarrow \infty} \vert d(x_{2m_i+1},x_{2n_i}) \vert = \vert c \vert$. Since $0 \precsim \xi(\alpha(x_{2m_i},x_{2n_i-1})d(x_{2m_i+1},x_{2n_i}),d(x_{2m_i},x_{2n_i-1}))$, $\lim\limits_{i \rightarrow \infty} \vert \alpha(x_{2m_i},x_{2n_i-1})d(x_{2m_i+1},x_{2n_i}) \vert =\vert c \vert$. Then $0 \precsim \xi(\vert\alpha(x_{2m_i},x_{2n_i-1})d(x_{2m_i+1},x_{2n_i})\vert,\vert d(x_{2m_i},x_{2n_i-1})\vert)\\ \precnsim 0$, a contradiction. Therefore, $\{x_{2n}\}$ is a Cauchy sequence in $(X,d)$. Then there exists $u \in X$ such that $\lim\limits_{n \rightarrow \infty} \vert d(x_n,u) \vert=0$. We observe that $$d(Su,u) \precsim d(Su,Tx_{2n_k+1})+d(Tx_{2n_k+1},u).$$ Since $X$ is $\alpha_{\mathbb{C}}$-regular, $1 \precsim \alpha(u,x_{2n_k+1})$ which implies $0\precsim \xi(\alpha(u,x_{2n_k+1}) d(Su,Tx_{2n_k+1}),d(u,x_{2n_k+1}))$. Then $\vert d(Su,Tx_{2n_k+1})\vert < \vert d(u,x_{2n_k+1})\vert$. Therefore, we have $\vert d(Su,u)\vert < \vert d(u,x_{2n_k+1})\vert+ \vert d(x_{2n_k+2},u)\vert $. Letting $i \rightarrow \infty$ we get, $Su=u$. Also, $$d(u,Tu) \precsim d(u,Sx_{2n_k})+d(Sx_{2n_k},Tu).$$ Since $X$ is $\alpha_{\mathbb{C}}$-regular, $ 1 \precsim \alpha(x_{2n_k},u)$ which gives $0\precsim \xi(\alpha(x_{2n_k},u) d(Sx_{2n_k},Tu), d(x_{2n_k},u))$. Therefore, $\vert d(Sx_{2n_k},Tu) \vert < \vert d(x_{2n_k},u)) \vert$. Therefore, we have $\vert d(u,Tu) \vert < \vert d(u,x_{2n_k+1}) \vert + \vert d(x_{2n_k},u) \vert$. Letting $i \rightarrow \infty$ we get, $Tu=u$. Thus, $u$ is a common fixed point of $S$ and $T$. Let $v$ be another common fixed point of $S$ and $T$. Consider $$0 \precsim \xi(\alpha(u,v)d(Su,Tv),d(u,v)) \precnsim d(u,v)-\alpha(u,v)d(u,v).$$ This implies that $\vert d(u,v) \vert \leq \vert \alpha(u,v)d(u,v) \vert < \vert d(u,v) \vert$, a contradiction. Hence, $S$ and $T$ have a unique common fixed point in $X$. 
\end{proof}
\begin{corollary}
Let $(X,d)$ be a complete complex valued metric space and $T:X \rightarrow X$ satisfies

(i) $0\precsim \alpha(x,y)d(Tx,Ty)$,

(ii) $\alpha(x,y)d(Tx,Ty) \precsim \lambda d(x,y)$ \\
for all $x,y \in X$, where $\lambda$ is a real number such that $0<\lambda <1$. Suppose that the following conditions hold:

(i) $T$ is triangular $\alpha_{\mathbb{C}}$-orbital admissible,

(ii) there exists $x_0 \in X$ such that $1 \precsim \alpha(x_0,Tx_0)$,

(iii) $X$ is $\alpha_{\mathbb{C}}$-regular.\\
Moreover, if for $x,y \in$ Fix$(T)$ we have $1\precsim \alpha(x,y)$. Then $T$ has a unique fixed point in $X$.
\end{corollary}
\begin{proof}
Choosing $\xi(t,s)= \lambda s-t$ for all $t,s \in \mathbb{S}$, where $\lambda$ is a real number such that $0< \lambda <1$ in Theorem \ref{theorem1} we have the desired result.
\end{proof}
It is worth noting that Banach contraction principle is an immediate consequence of Theorem \ref{theorem1}.
\begin{corollary}\label{corollary1}
\textbf{\emph{(Banach Contraction)}} Let $(X,d)$ be a complete complex valued metric space. Suppose that $T:X \rightarrow X$ satisfies
$$d(Tx,Ty)\precsim \lambda d(x,y)$$ for all $x,y \in X$, where $\lambda$ is a real number such that $0<\lambda <1$. Then $T$ has a unique fixed point in $X$.
\end{corollary}
\begin{corollary}
Let $(X,d)$ be a complete complex valued metric space and $T:X \rightarrow X$ satisfies

(i) $0\precsim \alpha(x,y)d(Tx,Ty)$,

(ii) $\phi(\alpha(x,y)d(Tx,Ty)) \precsim \lambda \psi(d(x,y))$ 

(iii) $\phi(\vert \alpha(x,y)d(Tx,Ty)\vert) \precsim \lambda \psi(\vert d(x,y) \vert)$\\
for all $x,y \in X$, where $\phi,\psi: \mathbb{S} \rightarrow \mathbb{S}$ are continuous functions, $\phi(t)=0=\psi(t)$ if and only if $t=0$ and $\psi(t)\precnsim t \precsim \phi (t)$ for all $0\precnsim t$. Suppose that the following conditions hold:

(i) $T$ is triangular $\alpha_{\mathbb{C}}$-orbital admissible,

(ii) there exists $x_0 \in X$ such that $1 \precsim \alpha(x_0,Tx_0)$,

(iii) $X$ is $\alpha_{\mathbb{C}}$-regular.\\
Moreover, if for $x,y \in$ Fix$(T)$ we have $1\precsim \alpha(x,y)$. Then $T$ has a unique fixed point in $X$.
\end{corollary}
\begin{proof}
 Take $\xi(t,s)=\psi(s)-\phi(t)$, where $\phi,\psi: \mathbb{S} \rightarrow \mathbb{S}$ are continuous functions, $\phi(t)=0=\psi(t)$ if and only if $t=0$ and $\psi(t)\precnsim t \precsim \phi (t)$ for all $0\precnsim t$ in Theorem \ref{theorem1}.
\end{proof}
\begin{example}
\emph{Let $X= \mathbb{C}$ and define $d:X \times X \rightarrow \mathbb{C}$ as $$d(z_1,z_2)=\vert x_1-y_1 \vert + i \vert x_2-y_2 \vert,$$ where $z_1=x_1 +ix_2$ and $z_2=y_1+iy_2$. Then $(X,d)$ is a complete complex valued metric space. Define $T:X \rightarrow X$ as $Tz=\frac{z+i}{2}$ and suppose that $\lambda$ is a real number such that $\frac{1}{2} \leq \lambda <1$. It is easily seen that $d(Tx,Ty) \precsim \lambda d(x,y)$ for all $x,y \in X$. Therefore, by Corollary \ref{corollary1} $T$ has a unique fixed point and the fixed point is $i$. }
\end{example}
The following example illustrates the use of Corollary  \ref{corollary1} in solving a nonlinear integral equation.
\begin{example}\label{example1}
\emph{Consider the nonlinear integral equation $$x(t)= 2+ \int_{a}^{t} (x(s)+s^3)e^{1-2s} ds, \quad  \quad t \in [a,b].$$  Let $X=C([a,b],\mathbb{R})$ be the set of continuous real valued functions defined on $[a,b]$. Suppose that $a>1$ and $b-a< e^{2a}-1$. Define $d:X \times X \rightarrow \mathbb{C}$ as $$d(z_1,z_2)= \max\limits_{t \in [a,b]} \vert z_1-z_2 \vert \frac{\sqrt{a^2+b^2}}{a} e^{i \taninv (\frac{b}{a} )}.$$ Define $T: X \rightarrow X$ as $$Tx(t)= 2+ \int_{a}^{t} (x(s)+s^3)e^{1-2s} ds, \quad  \quad t \in [a,b].$$ For $x,y \in X$ consider
\begin{align*}
\vert Tx(t)-Ty(t)\vert &= \Big\vert \int_{a}^{t}(x(s)-y(s))e^{1-2s} ds \Big\vert \\
& \leq \int_{a}^{t}\vert x(s)-y(s) \vert e^{1-2a} ds.
\end{align*}
Now 
\begin{align*}
d(Tx,Ty) & = \max\limits_{t \in [a,b]} \vert Tx(t)-Ty(t) \vert  \frac{\sqrt{a^2+b^2}}{a} e^{i \taninv (\frac{b}{a} )}\\
&\precsim \frac{b-a}{e^{2a-1}}d(x,y)
\end{align*}
Then for $\lambda=\frac{b-a}{e^{2a-1}}$ all the conditions of Corollary \ref{corollary1} are satisfied. Therefore, $T$ has a unique fixed point in $X$ which is the solution of the given nonlinear integral equation.} 
\end{example}
\begin{remark}
\emph{It is noted that from Example \ref{example1} we are able to find the solution of the following differential equation: $$\frac{dx}{dt}=(x+t^3)e^{1-2t}, \quad \quad t \in [1,2] \quad \mbox{and}\quad  x(1)=2 .$$ }
\end{remark}
We will give another application of Theorem \ref{theorem1} for finding common fixed points of two finite families of self-mappings.
\begin{definition}
\emph{(\cite{4}) Two families of self-mappings $\{S_i\}_{i=1}^n$ and $\{T_i\}_{i=1}^m$ are said to be pairwise commuting if the following conditions hold:}

\emph{(i) $S_i S_j=S_j S_i$, $i,j \in \{1,2,\ldots,n\}$,}

\emph{(ii) $T_iT_j=T_jT_i$, $i,j \in \{1,2,\ldots,m\}$,}

\emph{(iii) $S_iT_j=T_jS_i$, $i \in \{1,2,\ldots,n\}$ and $j \in \{1,2,\ldots,m\}$.}
\end{definition}
\begin{theorem}\label{theorem5}
Let $(X,d)$ be a complete complex valued metric space. Let $\{S_i\}_{i=1}^n$ and $\{T_i\}_{i=1}^m$ be a pairwise commuting finite families of self-mappings defined on $X$ such that $S=S_1S_2\ldots S_n$ and $T=T_1T_2\ldots T_m$. Suppose that the pair $(S,T)$ satisfies the hypothesis of Theorem \ref{theorem1}. Then the component maps of the two families $\{S_i\}_{i=1}^n$ and $\{T_i\}_{i=1}^m$ have a unique common fixed point in $X$.
\end{theorem}
\begin{proof}
By Theorem \ref{theorem1}, $S$ and $T$ have a unique common fixed point in $X$ say $u$. We observe that for every $k$, $S(T_ku)=T_k(Su)=T_ku$ and $T(T_ku)=T_k(Tu)=T_ku$. Similarly, $S_ku$ (for every $k$) is also a common fixed  point of $S$ and $T$. By the uniqueness of common fixed point we conclude that u is a common fixed point of the families $\{S_i\}_{i=1}^n$ and $\{T_i\}_{i=1}^m$.
\end{proof}
If we take $S_1=S_2=\ldots=S_n=S$ and $T_1=T_2=\ldots=T_m=T$  in the above theorem we deduce the following result:
\begin{corollary}
Let $(X,d)$ be a complete complex valued metric space and $S,T:X \rightarrow X$ be two commuting maps. Suppose that the pair $(S^n,T^m)$ satisfies the hypothesis of Theorem \ref{theorem1}. Then $S$ and $T$ have a unique common fixed point in $X$.
\end{corollary}
If we take $S=T$ and $n=m$ in Theorem \ref{theorem5} we have the following result:
\begin{corollary}
Let $(X,d)$ be a complete complex valued metric space and $T:X \rightarrow X$ be a map such that for all $x,y \in X$ the following conditions are satisfied:

(i) $0\precsim \alpha(x,y)d(T^nx,T^ny)$,

(ii) $0\precsim \xi(\alpha(x,y)d(T^nx,T^ny), d(x,y))$,

(iii) $0\precsim \xi(\vert \alpha(x,y)d(T^nx,T^ny)\vert, \vert d(x,y)\vert)$,

(iv) $T$ is triangular $\alpha_{\mathbb{C}}$-orbital admissible,

(v) there exists $x_0 \in X$ such that $1\precsim \alpha(x_0,Tx_0)$,

(vi) $X$ is $\alpha_{\mathbb{C}}$-regular.\\
Moreover, if for $x,y \in$ Fix$(T)$ we have $1\precsim \alpha(x,y)$. Then $T$ has a unique fixed point in $X$.
\end{corollary}
\begin{definition}
\emph{Let $(X,d)$ be a complex valued metric space, $S,T: X \rightarrow X$ be self-mappings and $\xi \in \mathscr{Z}_{\mathbb{C}}$. Then the pair $(S,T)$ is called a \emph{generalized $\alpha_{\mathbb{C}}$-admissible $\mathscr{Z}_{\mathbb{C}}$-contraction} with respect to $\xi$ if for all $x,y \in X$, the following conditions are satisfied:
\begin{enumerate}[(i)]
\item $0\precsim \alpha(x,y) d(Sx,Ty)$,
\item $0\precsim \xi(\alpha(x,y) d(Sx,Ty),M(x,y))$,\\
 where $M(x,y)=\lambda \max \Big\{\vert d(x,y)\vert,\vert d(x,Sx)\vert,\vert d(y,Ty)\vert,\frac{\vert d(x,Ty)\vert+\vert d(y,Sx)\vert}{2} \Big\}$ and $\lambda$ is a real number such that $0<\lambda <1$,
\item $0\precsim \xi(\vert \alpha(x,y) d(Sx,Ty) \vert, M(x,y) )$.
\end{enumerate} }
\end{definition}
\begin{theorem}\label{theorem2}
Let $(X,d)$ be a complete complex valued metric space and $S,T:X \rightarrow X$ be self-mappings such that the pair $(S,T)$ is a generalized $\alpha_{\mathbb{C}}$-admissible $\mathscr{Z}_{\mathbb{C}}$-contraction with respect to $\xi$. Suppose that 

(i) $(S,T)$ is triangular $\alpha_{\mathbb{C}}$-orbital admissible,

(ii) there exists $x_0 \in X$ such that $1 \precsim \alpha(x_0,Sx_0)$ and $1\precsim \alpha(Sx_0,x_0)$,

(iii) $X$ is $\alpha_{\mathbb{C}}$-regular.\\
Moreover, if for $x,y \in$ Fix$(T)$ $\cap$ Fix$(S)$ we have $1\precsim \alpha(x,y)$. Then $S$ and $T$ have a unique common fixed point in $X$.
\end{theorem}
\begin{proof}
 Proceeding as in Theorem \ref{theorem1} we construct a sequence $\{x_n\}$ in $X$ satisfying $\vert d(x_{2n+1},x_{2n+2}) \vert <  M(x_{2n},x_{2n+1})  $, where
$$M(x_{2n},x_{2n+1})= \lambda \max \Big\{\vert d(x_{2n},x_{2n+1})\vert,\vert d(x_{2n+1},x_{2n+2})\vert, \frac{\vert d(x_{2n},x_{2n+2})\vert}{2} \Big\}.$$
If $M(x_{2n},x_{2n+1})=\lambda \vert d(x_{2n},x_{2n+1})\vert$, then $$\vert d(x_{2n+1},x_{2n+2}) \vert < \lambda \vert d(x_{2n},x_{2n+1}) \vert .$$ If $M(x_{2n},x_{2n+1})=\lambda \vert d(x_{2n+1},x_{2n+2})\vert$, then $$\vert d(x_{2n+1},x_{2n+2}) \vert < \lambda \vert d(x_{2n+1},x_{2n+2}) \vert < \vert d(x_{2n+1},x_{2n+2}) \vert,$$ a contradiction. If $M(x_{2n},x_{2n+1})=\frac{\lambda}{2} \vert d(x_{2n},x_{2n+2})\vert$, then $$\vert d(x_{2n+1},x_{2n+2}) \vert < \frac{\lambda}{2} \vert d(x_{2n},x_{2n+2}) \vert < \frac{1}{2} \Big\{ \vert d(x_{2n},x_{2n+1})\vert + \vert d(x_{2n+1},x_{2n+2}) \vert \Big\}.$$ Therefore, from all the cases we have $\vert d(x_{2n+1},x_{2n+2}) \vert <  \vert d(x_{2n},x_{2n+1}) \vert $. Similarly, $\vert d(x_{2n+2},x_{2n+3}) \vert < \vert d(x_{2n+1},x_{2n+2}) \vert $. Then $\{\vert d(x_n,x_{n+1}) \vert \}$ is a monotonically non-increasing sequence of non-negative real numbers therefore, it is convergent. Let $\lim\limits_{n \rightarrow \infty} \vert d(x_n,x_{n+1}) \vert =L \geq 0$. Suppose that $L>0$. Then $\lim\limits_{n \rightarrow \infty} M(x_{2n},x_{2n+1})  = \lambda L$.  As $0\precsim \xi(\alpha(x_{2n},x_{2n+1})d(Sx_{2n},Tx_{2n+1}),M(x_{2n},x_{2n+1}))$, which gives $$\vert \alpha(x_{2n},x_{2n+1})d(Sx_{2n},Tx_{2n+1}) \vert  \leq  M(x_{2n},x_{2n+1}).$$ Therefore, $\lim\limits_{n \rightarrow \infty} \vert \alpha(x_{2n},x_{2n+1})d(Sx_{2n},Tx_{2n+1}) \vert =\lambda L$. Then $$0 \precsim \xi(\vert \alpha(x_{2n},\\x_{2n+1})d(Sx_{2n},Tx_{2n+1})\vert, M(x_{2n},x_{2n+1}))\precnsim 0,$$ a contradiction. Thus, $\lim\limits_{n \rightarrow \infty} \vert d(x_{n},x_{n+1}) \vert =0$. 

Now it suffices to show that $\{x_{2n}\}$ is a Cauchy sequence in $(X,d)$. On the contrary, assume that $\{x_{2n}\}$ is not a Cauchy sequence. Then there exists $c \in \mathbb{C}$ with $0\prec c$ for which we can find two subsequences $\{x_{2m_i}\}$ and $\{x_{2n_i}\}$ of $\{x_{2n}\}$ such that $n_i$ is the smallest positive integer with 
$$2n_i > 2m_i \geq i \quad \mbox{and} \quad c \precsim d(x_{2m_i},x_{2n_i}) .$$
This means that $ d(x_{2m_i},x_{2n_i-2})\prec c$. Following the lines in the proof of Theorem \ref{theorem1} we get, $\lim\limits_{i \rightarrow \infty} \vert d(x_{2m_i},x_{2n_i}) \vert =\lim\limits_{i \rightarrow \infty} \vert d(x_{2m_i},x_{2n_i-1}) \vert =\lim\limits_{i \rightarrow \infty} \vert d(x_{2m_i+1},x_{2n_i}) \vert= \vert c \vert$. Now 
\begin{align*}
M(x_{2m_i},x_{2n_i-1})& =\lambda \max \Big\{\vert d(x_{2m_i},x_{2n_i-1})\vert,\vert d(x_{2m_i},x_{2m_i+1})\vert,\vert d(x_{2n_i-1},x_{2n_i})\vert,\\
& \quad  \frac{\vert d(x_{2m_i},x_{2n_i})\vert+\vert d(x_{2n_i-1},x_{2m_i+1})\vert}{2}\Big\}.
\end{align*} If $M(x_{2m_i},x_{2n_i-1})=\lambda \vert d(x_{2m_i},x_{2n_i-1})\vert$, then $\lim\limits_{i \rightarrow \infty} M(x_{2m_i},x_{2n_i-1}) =\lambda \vert c \vert$. If $M(x_{2m_i},x_{2n_i-1})=\lambda \vert d(x_{2m_i},x_{2m_i+1})\vert $, then $\lim\limits_{i \rightarrow \infty}  M(x_{2m_i},x_{2n_i-1}) =0$. Since $0\precsim \xi(\alpha(x_{2m_i},x_{2n_i-1})d(x_{2m_i+1},x_{2n_i}),\\M(x_{2m_i},x_{2n_i-1}))$ then  $$\vert d(x_{2m_i+1},x_{2n_i}) \vert \leq \vert \alpha(x_{2m_i},x_{2n_i-1})d(x_{2m_i+1},x_{2n_i})\vert <  M(x_{2m_i},x_{2n_i-1}).$$ This gives $\vert c \vert \leq 0$, a contradiction. If $M(x_{2m_i},x_{2n_i-1})=\lambda \vert d(x_{2n_i-i},x_{2n_i})\vert$, then $\lim\limits_{i \rightarrow \infty} M(x_{2m_i},x_{2n_i-1}) =0$ and proceeding as in the previous case we get a contradiction.\\ If $M(x_{2m_i},x_{2n_i-1})=\frac{\lambda}{2}\Big\{ \vert d(x_{2m_i},x_{2n_i})\vert+\vert d(x_{2n_i-1},x_{2m_i+1})\vert \Big\}$, then $$ M(x_{2m_i},x_{2n_i-1}) \precsim \frac{\lambda}{2}\Big\{  \vert d(x_{2m_i},x_{2n_i})\vert + \vert d(x_{2n_i-1},x_{2n_i})\vert  + \vert d(x_{2n_i},x_{2m_i+1})\vert  \Big\}.$$ Letting $i \rightarrow \infty$ we get, $\lim\limits_{i \rightarrow \infty}  M(x_{2m_i},x_{2n_i-1}) \leq \lambda \vert c \vert$. Since $\lambda \vert d(x_{2m_i},x_{2n_i-1})\vert \leq M(x_{2m_i},x_{2n_i-1})$, $\lim\limits_{i \rightarrow \infty}  M(x_{2m_i},x_{2n_i-1}) =\lambda \vert c \vert$. Also, 
\begin{align*}
0& \precsim \xi(\alpha(x_{2m_i},x_{2n_i-1})d(x_{2m_i+1},x_{2n_i}),M(x_{2m_i},x_{2n_i-1}))\\
&\precnsim M(x_{2m_i},x_{2n_i-1})-\alpha(x_{2m_i},x_{2n_i-1})d(x_{2m_i+1},x_{2n_i}).
\end{align*}
 This gives $\vert d(x_{2m_i+1},x_{2n_i}) \vert \leq \vert \alpha(x_{2m_i},x_{2n_i-1})d(x_{2m_i+1},x_{2n_i}) \vert <  M(x_{2m_i},x_{2n_i-1}) $. Letting $i \rightarrow \infty$ we get, $$ \vert c \vert \leq \lim\limits_{i \rightarrow \infty} \vert \alpha(x_{2m_i},x_{2n_i-1})d(x_{2m_i+1},x_{2n_i}) \vert \leq \lambda \vert c \vert$$ which implies that $\lim\limits_{i \rightarrow \infty} \vert \alpha(x_{2m_i},x_{2n_i-1})d(x_{2m_i+1},x_{2n_i}) \vert = \lambda \vert c \vert$. Then $$0\precsim \xi(\vert \alpha(x_{2m_i},x_{2n_i-1})d(x_{2m_i+1},x_{2n_i}) \vert ,  M(x_{2m_i},x_{2n_i-1})) \precnsim 0,$$ a contradiction. Therefore, $\{x_{2n}\}$ is a Cauchy sequence in $(X,d)$. Then there exists $u \in X$ such that $\lim\limits_{n \rightarrow \infty} \vert d(x_n,u) \vert =0$. We observe that $$d(Su,u) \precsim d(Su,Tx_{2n_k+1})+d(Tx_{2n_k+1},u).$$ Since $X$ is $\alpha_{\mathbb{C}}$-regular, $1\precsim \alpha(u,x_{2n_k+1})$ which gives $0\precsim \xi(\alpha(u,x_{2n_k+1}) d(Su,Tx_{2n_k+1}),M(u,x_{2n_k+1}))$. Therefore, $\vert d(Su,Tx_{2n_k+1})\vert <  M(u,x_{2n_k+1})$, where \begin{align*}
 M(u,x_{2n_k+1})& =\lambda \max \Big\{\vert d(u,x_{2n_k+1})\vert,\vert d(u,Su)\vert,\vert d(x_{2n_k+1},x_{2n_k+2})\vert, \frac{\vert d(u,x_{2n_k+2})\vert+\vert d(x_{2n_k+1},Su)\vert}{2}\Big\}.
 \end{align*}
 \textbf{Case-1} If $M(u,x_{2n_k+1})=\lambda \vert d(u,x_{2n_k+1})\vert$, then $\vert d(Su,u) \vert < \lambda \vert d(u,x_{2n_k+1})\vert +\vert d(x_{2n_k+2},u)\vert$. Letting $k \rightarrow \infty$ we get, $Su=u$.\\
 \textbf{Case-2} If $M(u,x_{2n_k+1})=\lambda \vert d(u, Su)\vert$, then $\vert d(Su,u) \vert < \lambda \vert d(Su,u)\vert +\vert d(x_{2n_k+2},u)\vert$. Letting $k \rightarrow \infty$ we get, $\vert d(Su,u) \vert \leq \lambda \vert d(Su,u) \vert < \vert d(Su,u) \vert$, a contradiction. \\
 \textbf{Case-3} If $M(u,x_{2n_k+1})=\lambda \vert d(x_{2n_k+1},x_{2n_k+2})\vert$, then $\vert d(Su,u) \vert < \lambda \vert d(x_{2n_k+1},x_{2n_k+2})\vert +\vert d(x_{2n_k+2},u)\vert$. Letting $k \rightarrow \infty$ we get, $Su=u$.\\
 \textbf{Case-4} If $M(u,x_{2n_k+1})= \frac{\lambda}{2} \Big\{ \vert d(u,x_{2n_k+2})\vert+\vert d(x_{2n_k+1},Su)\vert \Big\}$, then $$\vert d(Su,u)\vert < \frac{\lambda}{2} \Big\{\vert d(u,x_{2n_k+2}) \vert+\vert d(x_{2n_k+1},u)\vert+\vert d(u,Su)\vert+\vert d(x_{2n_k+2},u)\vert \Big\}.$$ Letting $k \rightarrow \infty$ we get, $\vert d(Su,u)\vert <\vert d(Su,u)\vert$, a contradiction. Similarly, we can prove that $Tu=u$. Let $v$ be another common fixed point of $S$ and $T$. Consider $0\precsim \xi(\alpha(u,v)d(Su,Tv),M(u,v))$, where $M(u,v)= \lambda \vert d(u,v)\vert$. This implies that $\alpha(u,v) d(u,v) \precnsim M(u,v)$. Since $1 \precsim \alpha(u,v)$, $\vert d(u,v) \vert < \lambda \vert d(u,v) \vert < \vert d(u,v) \vert$, a contradiction. Hence, $S$ and $T$ have a unique common fixed point in $X$. 
\end{proof}
\begin{corollary}
Let $(X,d)$ be a complete complex valued metric space. Suppose that $T:X \rightarrow X$ satisfies

(i) $ \phi(d(Tx,Ty)) \precsim \psi(M(x,y))$,

(ii) $ \phi(\vert d(Tx,Ty)\vert ) \precsim \psi( M(x,y))$\\
for all $x,y \in X$, where $M(x,y)= \lambda \max \Big\{\vert d(x,y)\vert,\vert d(x,Tx)\vert,\vert d(y,Ty)\vert,\\\frac{\vert d(x,Ty)\vert+\vert d(y,Tx)\vert}{2}\Big\}$, $\lambda$ is a real number such that $0<\lambda<1$ and $\phi,\psi: \mathbb{S} \rightarrow \mathbb{S}$ are continuous functions, $\phi(t)=0=\psi(t)$ if and only if $t=0$ and $\psi(t)\precnsim t \precsim \phi (t)$ for all $0\precnsim t$. Then $T$ has a unique fixed point in $X$.
\end{corollary}
\begin{proof}
Take $\xi(t,s)=\psi(s)-\phi(t)$, where $\phi,\psi: \mathbb{S} \rightarrow \mathbb{S}$ are continuous functions, $\phi(t)=0=\psi(t)$ if and only if $t=0$ and $\psi(t)\precnsim t \precsim \phi (t)$ for all $0\precnsim t$ and $\alpha(x,y)=1$ for all $x,y \in X$ in Theorem \ref{theorem2}.
\end{proof}
As an application of Theorem \ref{theorem2} we state the following theorem:
\begin{theorem}
Let $(X,d)$ be a complete complex valued metric space and $\{S_i\}_{i=1}^n$ and $\{T_i\}_{i=1}^m$ be a pairwise commuting finite families of self-mappings defined on $X$ such that $S=S_1S_2\ldots S_n$ and $T=T_1T_2\ldots T_m$. Suppose that the pair $(S,T)$ satisfies the hypothesis of Theorem \ref{theorem2}. Then the component maps of the two families $\{S_i\}_{i=1}^n$ and $\{T_i\}_{i=1}^m$ have a unique common fixed point in $X$.
\end{theorem}
\begin{corollary}
Let $(X,d)$ be a complete complex valued metric space and $T:X \rightarrow X$ be a map such that for all $x,y \in X$ the following conditions are satisfied:

(i) $0\precsim \alpha(x,y)d(T^nx,T^ny)$,

(ii) $0\precsim \xi(\alpha(x,y)d(T^nx,T^ny), M(x,y))$, 

where $M(x,y)= \lambda \max \Big\{\vert d(x,y)\vert,\vert d(x,Tx)\vert,\vert d(y,Ty)\vert,\frac{\vert d(x,Ty)\vert+\vert d(y,Tx)\vert}{2}\Big\}$ and $\lambda$ is a real number such that $0<\lambda<1$,

(iii) $0\precsim \xi(\vert \alpha(x,y)d(T^nx,T^ny)\vert,  M(x,y))$,

(iv) $T$ is triangular $\alpha_{\mathbb{C}}$-orbital admissible,

(v) there exists $x_0 \in X$ such that $1\precsim \alpha(x_0,Tx_0)$,

(vi) $X$ is $\alpha_{\mathbb{C}}$-regular.\\
Moreover, if for $x,y \in$ Fix$(T)$ we have $1\precsim \alpha(x,y)$. Then $T$ has a unique fixed point in $X$.
\end{corollary}
\begin{theorem}\label{theorem3}
Let $(X,d)$ be a complete complex valued metric space, $\xi \in \mathscr{Z}_{\mathbb{C}}$ and $S,T:X \rightarrow X$ be a pair of self-mappings such that for all $x,y \in X$, the following conditions hold: 

(i) $0\precsim \alpha(x,y)d(Sx,Ty)$,

(ii) $0\precsim \xi(\alpha(x,y)d(Sx,Ty), N(x,y))$, 

where $N(x,y)= \max \Big\{\vert d(x,y)\vert,\frac{\vert d(x,Sx)\vert \vert d(y,Ty)\vert+\vert d(x,Ty)\vert \vert d(y,Sx)\vert}{1+\vert d(x,y)\vert}\Big\}$,

(iii) $0\precsim \xi(\vert \alpha(x,y)d(Sx,Ty)\vert, N(x,y))$,

(iv) $(S,T)$ is triangular $\alpha_{\mathbb{C}}$-orbital admissible,

(v) there exists $x_0 \in X$ such that $1 \precsim \alpha(x_0,Sx_0)$ and $1\precsim \alpha(Sx_0,x_0)$,

(vi) $X$ is $\alpha_{\mathbb{C}}$-regular.\\
Moreover, if for $x,y \in$ Fix$(T)$ $\cap$ Fix$(S)$ we have $1\precsim \alpha(x,y)$. Then $S$ and $T$ have a unique common fixed point in $X$.
\end{theorem}
\begin{proof}
Following the lines of the proof of Theorem \ref{theorem1} we construct a sequence $\{x_n\}$ in $X$ satisfying $\vert d(x_{2n+1},x_{2n+2}) \vert <  N(x_{2n},x_{2n+1})$, where
$$N(x_{2n},x_{2n+1})=\max \Big\{\vert d(x_{2n},x_{2n+1})\vert, \frac{\vert d(x_{2n},x_{2n+1})\vert \vert d(x_{2n+1},x_{2n+2})\vert}{1+\vert d(x_{2n},x_{2n+1})\vert}\Big\}.$$
If $N(x_{2n},x_{2n+1})=\frac{\vert d(x_{2n},x_{2n+1})\vert \vert d(x_{2n+1},x_{2n+2})\vert}{1+\vert d(x_{2n},x_{2n+1})\vert}$, then $$\vert d(x_{2n+1},x_{2n+2}) \vert < \frac{\vert d(x_{2n},x_{2n+1})\vert \vert d(x_{2n+1},x_{2n+2})\vert}{\vert 1+d(x_{2n},x_{2n+1})\vert} < \vert d(x_{2n+1},x_{2n+2})\vert,$$ a contradiction. This gives $\vert d(x_{2n+1},x_{2n+2}) \vert < \vert d(x_{2n},x_{2n+1}) \vert$. Similarly, $\vert d(x_{2n+3},x_{2n+2}) \vert < \vert d(x_{2n+2},\\x_{2n+1}) \vert$. Then $\{\vert d(x_n,x_{n+1}) \vert \}$ is a monotonically non-increasing sequence of non-negative real numbers therefore, it is convergent. It is easily seen that $\lim\limits_{n \rightarrow \infty} \vert d(x_n,x_{n+1}) \vert =0$. It suffices to prove that $\{x_{2n}\}$ is a Cauchy sequence in $(X,d)$. On the contrary, assume that $\{x_{2n}\}$ is not a Cauchy sequence in $(X,d)$. Then there exists $c \in \mathbb{C}$ with $0\prec c$ for which we can find two subsequences $\{x_{2m_i}\}$ and $\{x_{2n_i}\}$ of $\{x_{2n}\}$ such that $n_i$ is the smallest positive integer with 
$$2n_i > 2m_i \geq i \quad \mbox{and} \quad c \precsim d(x_{2m_i},x_{2n_i}) .$$
This means that $ d(x_{2m_i},x_{2n_i-2})\prec c$. Following the lines in the proof of Theorem \ref{theorem1} we get, $\lim\limits_{i \rightarrow \infty} \vert  d(x_{2m_i},x_{2n_i}) \vert =\lim\limits_{i \rightarrow \infty} \vert d(x_{2m_i},x_{2n_i-1}) \vert=\lim\limits_{i \rightarrow \infty} \vert d(x_{2m_i+1},\\x_{2n_i}) \vert= \vert c \vert$. Now 
\begin{align*}
N(x_{2m_i},x_{2n_i-1})&=\max \Big\{\vert d(x_{2m_i},x_{2n_i-1})\vert,\frac{1}{1+\vert d(x_{2m_i},x_{2n_i-1})\vert}\{\vert d(x_{2m_i},x_{2m_i+1})\vert \\
& \quad  \vert d(x_{2n_i-1},x_{2n_i})\vert+\vert d(x_{2m_i},x_{2n_i})\vert \vert d(x_{2n_i-1},x_{2m_i+1})\vert\}  \Big\}.
\end{align*} If $N(x_{2m_i},x_{2n_i-1})=\frac{\vert d(x_{2m_i},x_{2m_i+1})\vert \vert d(x_{2n_i-1},x_{2n_i})\vert+\vert d(x_{2m_i},x_{2n_i})\vert \vert d(x_{2n_i-1},x_{2m_i+1})\vert}{1+\vert d(x_{2m_i},x_{2n_i-1})\vert}$, then
\begin{align*}
 N(x_{2m_i},x_{2n_i-1})& \leq \frac{1}{\vert 1+ d(x_{2m_i},x_{2n_i-1})\vert}\Big[\vert d(x_{2m_i},x_{2m_i+1})\vert \vert d(x_{2n_i-1},x_{2n_i})\vert\\
& \quad +\{\vert d(x_{2m_i},x_{2n_i-1})\vert + \vert d(x_{2n_i-1},x_{2n_i})\vert \}\vert d(x_{2n_i-1},x_{2m_i+1})\vert \Big]\\
& < \frac{\vert d(x_{2m_i},x_{2m_i+1})\vert \vert d(x_{2n_i-1},x_{2n_i})\vert}{\vert 1+ d(x_{2m_i},x_{2n_i-1})\vert}+  \vert d(x_{2n_i-1},x_{2m_i+1}) \vert \\
& \quad + \frac{\vert d(x_{2n_i-1},x_{2n_i})\vert \vert d(x_{2n_i-1},x_{2m_i+1})\vert}{\vert 1+ d(x_{2m_i},x_{2n_i-1})\vert}\\
& \leq \frac{\vert d(x_{2m_i},x_{2m_i+1})\vert \vert d(x_{2n_i-1},x_{2n_i})\vert}{\vert 1+ d(x_{2m_i},x_{2n_i-1})\vert}+\vert d(x_{2n_i-1},x_{2n_i}) \vert\\
&\quad  +  \vert d(x_{2n_i},x_{2m_i+1}) \vert+\frac{\vert d(x_{2n_i-1},x_{2n_i})\vert \vert d(x_{2n_i-1},x_{2m_i+1})\vert}{\vert 1+ d(x_{2m_i},x_{2n_i-1})\vert}.
\end{align*}
Letting $i \rightarrow \infty$ we get, $\lim\limits_{i \rightarrow \infty} N(x_{2m_i},x_{2n_i-1})\leq  \vert c \vert$. Also, $\vert d(x_{2m_i},x_{2n_i-1})\vert \leq N(x_{2m_i},x_{2n_i-1})$. Therefore, $\lim\limits_{i \rightarrow \infty} N(x_{2m_i},x_{2n_i-1})= \vert c \vert$. Since 
\begin{align*}
0 & \precsim \xi(\alpha(x_{2m_i},x_{2n_i-1})d(x_{2m_i+1},x_{2n_i}),N(x_{2m_i},x_{2n_i-1}))\\
&\precnsim N(x_{2m_i},x_{2n_i-1})-\alpha(x_{2m_i},x_{2n_i})d(x_{2m_i+1},x_{2n_i})
\end{align*} which gives $\lim\limits_{n \rightarrow \infty} \vert \alpha(x_{2m_i},x_{2n_i-1})d(x_{2m_i+1},x_{2n_i})\vert =\vert c \vert$. So, $$0\precsim \xi(\vert \alpha(x_{2m_i},x_{2n_i-1})d(x_{2m_i+1},x_{2n_i})\vert , N(x_{2m_i},x_{2n_i-1}))\precnsim 0,$$ a contradiction. Thus, $\{x_{2n}\}$ is a Cauchy sequence in $(X,d)$. Then there exists $u \in X$ such that $\lim\limits_{n \rightarrow \infty} \vert d(x_n,u) \vert=0$. Consider $d(Su,u) \precsim d(Su,Tx_{2n_k+1})+d(Tx_{2n_k+1},u)$. As $0\precsim \xi(\alpha(u,x_{2n_k+1})\\d(Su,Tx_{2n_k+1}),N(u,x_{2n_k+1}))$ which implies that $\vert d(Su,Tx_{2n_k+1})\vert < N(u,x_{2n_k+1})) $, where
 \begin{align*}
 N(u,x_{2n_k+1})&=\max \Big\{\vert d(u,x_{2n_k+1})\vert, \frac{\vert d(u,Su)\vert \vert d(x_{2n_k+1},x_{2n_k+2})\vert+\vert d(u,x_{2n_k+2})\vert \vert d(x_{2n_k+1},Su)\vert}{1+\vert d(u,x_{2n_k+1})\vert}\Big\}.
 \end{align*} 
 \textbf{Case-1} If $N(u,x_{2n_k+1})= \vert d(u,x_{2n_k+1})\vert$, then $\vert d(Su,u)\vert < \vert d(u,x_{2n_k+1})\vert+\vert d(x_{2n_k+2},u)\vert$. Letting $k \rightarrow \infty$ we get, $Su=u$.\\
 \textbf{Case-2} If $N(u,x_{2n_k+1})= \frac{\vert d(u,Su)\vert \vert d(x_{2n_k+1},x_{2n_k+2})\vert+\vert d(u,x_{2n_k+2})\vert \vert d(x_{2n_k+1},Su)\vert}{1+\vert d(u,x_{2n_k+1})\vert}$, then 
 \begin{align*}
 \vert d(Su,u)\vert & < \frac{\vert d(u,Su)\vert \vert d(x_{2n_k+1},x_{2n_k+2})\vert +\vert d(u,x_{2n_k+2})\vert \vert d(x_{2n_k+1},Su)\vert}{\vert 1+d(u,x_{2n_k+1})\vert} +\vert d(x_{2n_k+2},u)\vert.
 \end{align*} Letting $k \rightarrow \infty$ we get, $Su=u$. Similarly, we can prove that $Tu=u$. Let $v$ be another common fixed point of $S$ and $T$. Consider $$0\precsim \xi(\alpha(u,v)d(Su,Tv),N(u,v)),$$ where $N(u,v)= \max \Big\{ \vert d(u,v)\vert,\frac{{\vert d(u,v)\vert}^2}{1+\vert d(u,v)\vert} \Big\}$. In both the cases we get, $\vert d(u,v) \vert  < \vert d(u,v) \vert$, a contradiction. Hence, $S$ and $T$ have a unique common fixed point in $X$. 
\end{proof}
\section{Application to Differential Equations}
In this section, inspired by Natashi and Vetro \cite{6} and Harjani and Sadarangani \cite{3} we establish the existence of a solution of a differential equation as an application of our result.

Let $X=C([0,a], \mathbb{R}^n)$ be the space of continuous functions $u: [0,a] \rightarrow \mathbb{R}^n$. Let $\Vert (u_1,u_2,\ldots,u_n )\Vert= \max \{\vert u_1 \vert, \vert u_2 \vert, \ldots , \vert u_n \vert \}$ and define $d:X \times X \rightarrow \mathbb{C}$ as $$d(u,v)=\max\limits_{t \in [0,a]} \Vert u(t)-v(t) \Vert \sqrt{1+a^2} e^{i \taninv a}$$ for all $u, v \in X$. Then $(X,d)$ is a complete complex valued metric space. Consider the first-order periodic problem
\begin{equation}\label{equation3.1}
\begin{split}
u'(t)&=f(t,u(t)),\quad  t \in [0,a],\\
u(0)&=u(a),
\end{split}
\end{equation} 
where $f:[0,a] \times \mathbb{R}^n \rightarrow \mathbb{R}^n$ is a continuous function. This is equivalent to
\begin{equation}\label{equation3.2}
\begin{split}
u'(t)+ \eta u(t)&=f(t,u(t))+\eta u(t),\quad  t \in [0,a], \eta >1 \thinspace \thinspace \mbox{is a real number}\\\
u(0)&=u(a), 
\end{split}
\end{equation}
 Then equation (\ref{equation3.2}) is equivalent to the following integral equation: 
 $$u(t)= \int_0^a H(t,s) [f(s,u(s))+\eta u(s)]ds,$$ where 
 $$H(t,s)=\left\{\begin{array}{cc}
\frac{ e^{\eta(a+s-t)}}{e^{\eta a}-1},& \mbox{if}\thinspace \thinspace 0 \leq s \leq t \leq a,\\
\frac{ e^{\eta(s-t)}}{e^{\eta a}-1}, &\mbox{if} \thinspace \thinspace 0 \leq t \leq s \leq a.
\end{array}
\right. $$ 
Note that $\int_0^a H(t,s)ds = \frac{1}{\eta}$. 
\begin{theorem}\label{theorem4}
Consider  the first-order periodic problem (\ref{equation3.1}) with $f:[0,a] \times \mathbb{R}^n \rightarrow \mathbb{R}^n$  a continuous function and suppose that there exists $\eta >1$ such that $$\Vert f(t,u)+\eta u -f(t,v)- \eta v\Vert  \leq \Vert u-v \Vert $$ for all $u,v \in X$. Then equation (\ref{equation3.1}) has a unique solution.
\end{theorem}
\begin{proof}
Define $T: X \rightarrow X$ as 
$$(Tu)(t)= \int_0^a H(t,s) [f(s,u(s))+\eta u(s)]ds.$$ Observe that $u \in X$ is a fixed point of $T$ if and only if $u$ is a solution of (\ref{equation3.1}). Since $G(t,s)>0$ for $t,s \in [0,a]$, for every $u,v \in X$ we have
\begin{align*}
d(Tu,Tv) & = \max\limits_{t \in [0,a]} \Vert (Tu)(t)-(Tv)(t) \Vert \sqrt{1+a^2} e^{i \taninv a} \\
 &= \max\limits_{t \in [0,a]} \Big\Vert \int_0^a H(t,s)[ f(s,u(s))+ \eta u(s) -f(s,v(s))- \eta v(s)]ds \Big\Vert \\
 & \quad \sqrt{1+a^2} e^{i \taninv a}\\
 & \precsim  \max\limits_{t \in [0,a]} \int_0^a H(t,s)\Vert  u(s)- v(s)\Vert \sqrt{1+a^2}  e^{i \taninv a}ds\\
 & \precsim \frac{1}{\eta} d(u,v).
\end{align*}
Therefore, all the conditions of Corollary \ref{corollary1} are satisfied. Hence, $T$ has a unique fixed point in $X$
\end{proof}
\begin{example}
\emph{Consider the first-order periodic problem
\begin{equation}\label{equation3.3}
\begin{split}
u'(t)&=f(t,u(t)),\quad  t \in [0,1],\\
u(0)&=u(1),
\end{split}
\end{equation}
where $f:[0,1] \times \mathbb{R}^n \rightarrow \mathbb{R}^n$ is defined as $f(t,u)=(t-u_1,t-u_2,\ldots , t-u_n)$ for all $t \in [0,1]$ and $u=(u_1,u_2,\ldots,u_n)$. For $1< \eta <2$ , $f(t,u)+\eta u-f(t,v)- \eta v= ((\eta-1)(u_1-v_1),(\eta-1)(u_2-v_2),\ldots, (\eta-1)(u_n-v_n))$. Then $\Vert f(t,u)+\eta u-f(t,v)- \eta v \Vert \leq \Vert u-v \Vert$. Using Theorem \ref{theorem4}, equation (\ref{equation3.3}) has a unique solution.} 
\end{example}
\begin{example}
\emph{Consider the first-order periodic problem
\begin{equation}\label{equation3.4}
\begin{split}
u'(t)&=f(t,u(t)),\quad  t \in [0,2],\\
u(0)&=u(1),
\end{split}
\end{equation}
where $f:[0,2] \times \mathbb{R}^n \rightarrow \mathbb{R}^n$ is defined as $f(t,u)=-\ln(10+t^2)(u_1,u_2,\ldots , u_n)$ for all $t \in [0,2]$ and $u=(u_1,u_2,\ldots,u_n)$. For $-1+\ln 10  < \eta <1+ \ln 14$ , $f(t,u)+\eta u-f(t,v)- \eta v= (\eta-\ln(10+t^2))((u_1-v_1),(u_2-v_2),\ldots, (u_n-v_n))$. Then $\Vert f(t,u)+\eta u-f(t,v)- \eta v \Vert \leq \Vert u-v \Vert$. Using Theorem \ref{theorem4}, equation (\ref{equation3.4}) has a unique solution.}
\end{example}
\section*{Acknowledgements}
 The corresponding author(Manu Rohilla) is supported by UGC Non-NET fellowship (Ref.No. Sch/139/Non-NET/Math./Ph.D./2017-18/1028).

\end{document}